\documentclass[12pt]{article}
\usepackage{amsmath,amssymb,amsthm}  				    
\usepackage{fullpage,lineno}

\begin{document}

\newtheorem{theorem}{Theorem}
\newtheorem{lemma}[theorem]{Lemma}
\newtheorem{corollary}[theorem]{Corollary}
\newtheorem{prop}[theorem]{Proposition}
\theoremstyle{definition}
\newtheorem{example}{Example}
\newtheorem{remark}[theorem]{Remark}

\newcommand\esub{\subseteq}
\newcommand\sub{\subset}
\newcommand\nul{\varnothing} 

\def\Xb{\overline{X}}
\def\Yb{\overline{Y}}
\def\Tb{\overline{T}}
\def\Sb{\overline{S}}
\def\st{\colon\,}
\def\nul{\varnothing}


\title{Strongly connectable digraphs and non-transitive dice}

\author{
Simon Joyce\thanks{Binghamton University (SUNY), {\tt joyce@math.binghamton.edu}}\,,
Alex Schaefer\thanks{Binghamton University (SUNY), {\tt schaefer@math.binghamton.edu}}\,,
Douglas B.\ West\thanks{Departments of Mathematics, Zhejiang Normal University
\& University of Illinois,\newline {\tt dwest@math.uiuc.edu}.} \thanks{Research supported by
Recruitment Program of Foreign Experts, 1000 Talent Plan, State Administration
of Foreign Experts Affairs, China.
}\,,
and 
Thomas Zaslavsky\thanks{Binghamton University (SUNY), {\tt zaslav@math.binghamton.edu}}
}
\date{\today}
\maketitle

\begin{abstract}
We give a new proof of the theorem of Boesch--Tindell and
Farzad--Mahdian--Mahmoodian--Saberi--Sadri that a directed graph extends to a
strongly connected digraph on the same vertex set if and only if it has no
complete directed cut.  Our proof bounds the number of edges needed for such an
extension; we give examples to demonstrate sharpness.  We
apply the characterization to a problem on non-transitive dice.
\end{abstract}

{\small

}

\section{Introduction}

Characterization theorems in graph theory are important because they
often guarantee short certificates for both a ``yes'' and ``no'' answer to the 
question of whether a graph satisfies a particular property.
Many such theorems state that obvious necessary conditions are also sufficient.
In this note we give a new proof of such a result about directed graphs (Theorem \ref{main}, originally due to Farzad, Mahdian, Mahmoodian, Saberi, and Sadri~\cite{FMMSS} based on a similar theorem of Boesch and Tindell~\cite{BT}) and we apply it to a problem on non-transitive dice from Schaefer~\cite{S}.  
Our proof applies to a less general family than the original but yields a
previously unknown numerical bound.

A \emph{strict digraph} is a directed graph in which each
unordered pair of vertices is the set of endpoints of at most one edge;
that is, a strict digraph is an orientation of a simple graph.
A digraph is \emph{strongly connected}, or \emph{strong}, if it contains a (directed) path from
$x$ to $y$ for every ordered pair of vertices.  A digraph $G$ \emph{extends} to
a digraph $G'$ if $V(G)=V(G')$ and $E(G)\esub E(G')$.  We ask when a (strict)
digraph extends to a strongly connected strict digraph.  Note that it does so
if and only if it extends to a strongly connected tournament, 
where a \emph{tournament} is an orientation of a complete graph.  

For a set $X\esub V(G)$, let $\Xb = V(G)-X$.  When $\nul\ne X\subset V(G)$,
the \emph{cut} $[X,\Xb]$ is the set $\{xy\in E(G)\st x\in X,y\in\Xb\}$, where
we write $xy$ for an edge oriented from $x$ to $y$.  A cut $[X,\Xb]$ is a
\emph{dicut} if there is no ``back edge'' of the form $yx$ with $y\in\Xb$ and
$x\in X$.  It is a \emph{complete dicut} if it contains all $|X|\cdot|\Xb|$
edges of the form $xy$ with $x\in X$ and $y\in\Xb$.

Obviously, a digraph that extends to a strongly connected strict digraph
contains no complete dicut.  This obvious necessary condition is also
sufficient.  That fact is a special case of a theorem by Farzad et al.~\cite[Theorem B(iii)]{FMMSS}.

\begin{theorem}[{Strong connectability}]\label{main}
A strict digraph of order at least $3$ extends to a strong strict digraph if and only if
it contains no complete dicut.
\end{theorem}

We call such a digraph \emph{strongly connectable}.  We prove sufficiency in
Section~\ref{scd} from a new, stronger result, Theorem~\ref{ubound}, in which
we give an upper bound on how many edges need to be added.  

A consequence of Theorem \ref{main} is that the problem of deciding 
strong connectability belongs to the class NP $\cap$
co-NP.  A certificate for strong connectability is a strongly connected
extension, and a certificate for not being strongly connectable is a complete
dicut.  Both are easy to confirm.  The latter is obvious.  For the former,
there exists an extension to a strong tournament, which
contains a spanning cycle by Moon~\cite[Section 4]{M}.

\medskip

In Section~\ref{secdice} we apply Theorem~\ref{main} to a problem on
``non-transitive dice'' discussed in Schaefer~\cite{S}.  A set of dice is
\emph{non-transitive} if there is a cycle such that each die beats the next
in cyclic order, and it is \emph{balanced} if there exists a value $p>1/2$ such
that for any two dice, one beats the other with probability exactly $p$.
Schaefer asked which digraphs are realizable by balanced non-transitive dice;
he showed that those digraphs are precisely the ones that are strongly
connectable.  (That result led to our investigation of strong connectability.)  
Thus Theorem \ref{main} provides a criterion for a digraph to be realizable by
balanced non-transitive dice, where a set of dice \emph{realizes} a digraph if
the vertices can be assigned to dice in the set so that when $uv$ is an edge,
the die representing $u$ beats the die representing $v$. 

\medskip

The theorem of Farzad et al.\ is based on a theorem of Boesch and 
Tindell~\cite{BT}.  Consider a \emph{partially oriented multigraph} $M$, that is, a
graph that may have multiple edges (but not loops) and in which a subset of
edges has been oriented.  Assume the \emph{underlying graph} $\hat M$, which is
obtained from $M$ by treating all edges as undirected, is connected.
A \emph{weak dipath} in $M$ is a subgraph $P$ of $M$ whose underlying graph
$\hat P$ is a path in $\hat M$ such that every oriented edge in $P$ follows the
same direction as $\hat P$; that is, there are no ``back edges'', although $P$
may contain undirected edges.

We call $M$ \emph{strong} if for every ordered pair $(u,v)$ of vertices, there
is a weak dipath in $M$ from $u$ to $v$.  An \emph{$M$-cut} is a set of the
form $\{xy\in E(M)\st xy\text{ is oriented},\ x\in X,\ y\in\Xb\}$ for some nonempty
$X\sub V(M)$; we call $X$ the \emph{originating set}.  An \emph{$M$-dicut} is an
$M$-cut that contains every edge of $M$ having exactly one endpoint in its
originating set.  That is, for an $M$-dicut where $X$ is the originating set,
$M$ has no undirected edge $xy$ or back edge $yx$ with $x\in X$ and $y\in \Xb$.
In particular, when $\hat M=K_n$ and $G$ is the subdigraph of $M$ consisting of
all oriented edges, the $M$-dicuts are precisely the complete dicuts of $G$ as
defined before Theorem~\ref{main}.  

Boesch and Tindell characterized when $M$ can be oriented (by directing the
undirected edges) to be strongly connected.  Their result generalizes Robbins'
Theorem~\cite{R}, which solved the special case where all edges are unoriented. 
Farzad et al.~\cite[Theorem B(i, iii)]{FMMSS} observed an important corollary:
another necessary and sufficient condition is that the underlying graph
$\hat M$ be 2-edge-connected and there be no $M$-dicut.  (Farzad et al.~\cite{FMMSS}
attribute this to Boesch and Tindell~\cite{BT}, but it was not stated as such in \cite{BT}.)  This result
is more general than Theorem \ref{main}, which considers only the special case
where $\hat M$ is the complete graph $K_n$, with our strict digraph being their
subgraph of oriented edges in $M$.  It is not clear whether our proof bounding
the number of edges of $K_n$ that must be oriented generalizes to their setting.

\medskip

A different analogue of our problem has also been studied previously.  Eswaran
and Tarjan~\cite{ET} studied strong connectability for general digraphs, which
allow antiparallel pairs of edges.  This makes a huge difference.  Without
strictness, every digraph extends to a strong digraph, simply by introducing
every ordered pair of distinct vertices as an edge.  Thus, their tasks are to
find the minimum number of edges to add and an algorithm to produce a smallest
strong extension.  This can be viewed as another special case of the
Boesch--Tindell model, in which $\hat M$ is $K_n$ with all edges doubled and
$M$ has no parallel directed edges.

Frank~\cite{F} and Frank and Jord\'an~\cite{FJ1, FJ2}
generalized the questions to extensions that have connectivity or
edge-connectivity at least $k$, again allowing antiparallel edge pairs and
again seeking the smallest such extension and an algorithm.  

The problems for strict digraphs are very different, beginning with the fact
that not every strict digraph is strongly connectable.
We also do not find an exact minimum number or smallest set of edges.  
The complexity of determining the minimum size of an extension set remains
open, and it seems hard to generalize Theorem~\ref{main} to
characterize digraphs extendable to a strict $k$-connected digraph.

\section{Strongly connectable digraphs}\label{scd}

A \emph{strong component} of a digraph $G$ is a maximal strongly connected
subgraph.  The strong components of $G$ yield an acyclic digraph $G^*$ by
contracting each strong component to a vertex and eliminating duplicate edges
and loops.  A strong component is a \emph{source} or \emph{sink component} of
$G$ according as it is a source or sink vertex in $G^*$ (it may be both).
Given distinct strong components $C$ and $C'$, we say that $C'$ is a
\emph{successor} of $C$ if there is a path from $C$ to $C'$ in $G^*$, and then
also $C$ is a \emph{predecessor} of $C'$.  For $v\in V(G)$, let $C(v)$ denote
the strong component of $G$ containing $v$.  The \emph{underlying graph} of a
digraph $G$, as with a partially oriented multigraph, is the graph $\hat G$
obtained by treating its edges as unordered pairs.  We say that $G$ is
\emph{weakly connected} when $\hat G$ is connected.  The \emph{weak components}
of a digraph $G$ are the subdigraphs induced by the vertex sets of components
of $\hat G$.

We prove Theorem~\ref{main} by obtaining an upper bound on the number of edges
needed.  
We will strengthen the upper bound for disconnected graphs in
Lemma~\ref{dubound-sum} and Proposition~\ref{dubound}.

\begin{theorem}[Upper bound]\label{ubound}
Let $G$ be a strict digraph having at least three vertices and $r$ strong
components.  If $G$ has no complete dicut, then $G$ extends to a strongly
connected strict digraph by adding at most $r$ edges, with equality if and only
if $\hat G$ is disconnected and each weak component of $G$ is strong.
\end{theorem}

\begin{proof}
Suppose that $\hat G$ is disconnected and every weak component of $G$ is strong.
If $r>2$, then choose one vertex from each component and add a cycle of $r$
edges through them to obtain a strongly connected extension.  If $r=2$, then
since $G$ has at least three vertices, one weak component has at least two
vertices, and we can add edges to and from distinct vertices in that component
to obtain a strong extension.  Since every weak component must receive an added
entering edge, equality holds in this case.

Henceforth we may assume that not every weak component is strong, so $G$ has
a strong component that is a sink component but not a source component.  

{\bf Case 1:} \emph{$G$ is weakly connected.}
To prove the upper bound $r-1$, we use induction on $r$.  For $r=1$ there is
nothing to prove (no edge need be added).  Consider $r>1$.

Let $S$ be the set of all vertices in source components of $G$.  Since
$[S,\Sb]$ is a dicut, there is a pair $(y,x)$ with $y\in S$ and $x\in \Sb$
such that $yx\notin E(G)$.  Add edge $xy$ to $G$, forming a new strict digraph
$G'$.  If $C(x)$ is a successor to $C(y)$ in $G$, then $G'$ has a strong
component containing $C(x)$ and $C(y)$, so $G'$ has at most $r-1$ strong
components.  Otherwise, $x\notin S$ implies that $C(x)$ has some source
component as a predecessor; let $z$ be a vertex in that component, and let
$G''=G'+yz$.  Since no edge connects two source components, $G''$ is a strict
digraph.  It has at most $r-2$ strong components, since $C(x)$, $C(y)$, and
$C(z)$ all lie in a single strong component of $G''$.  It now suffices by the
induction hypothesis to show that $G'$ and in the second case also $G''$ has no
complete dicut.

Let $[X,Y]$ be a complete dicut in $G'$.  As $G$ has no complete dicut, the
added edge $xy$ must satisfy $x\in X$ and $y\in Y$.  Thus $C(x)\esub X$ and
$C(y)\esub Y$.  In $G$ no edges enter the source component $C(y)$, so in $G'$
only one edge enters $C(y)$.  Since $C(y)\esub Y$, this implies
$|V(C(y))|=1$.  That makes $y$ a source vertex in $G$, so only the edge $xy$
enters it in $G'$.  Thus $|X|=1$.  This implies that $x$ is a source vertex in
$G'$ and therefore in $G$, which contradicts $x\notin S$.  We conclude that
$G'$ has no complete dicut.  

In the case where $C(x)$ is not a successor to $C(y)$, the source component
$C(z)$ in $G$ remains a source component in $G'$.  Also, the strong component
of $G'$ that contains both $C(x)$ and $C(y)$ from $G$ is a successor of $C(z)$
in $G'$.  Thus $G''$ is formed from $G'$ by adding $yz$ in the way that $G'$
was formed from $G$ by adding $xy$.  Since $G'$ satisfies the same hypotheses
required of $G$, the same argument now implies that $G''$ also has no complete
dicut.

{\bf Case 2:} \emph{$G$ is not weakly connected.}
Let $G_1,\ldots,G_k$ with $k>1$ be the weak components of $G$.  In each $G_i$
choose a source component $S_i$ and a sink component $T_i$ such that $T_i=S_i$
if $G_i$ is strong and otherwise $T_i$ is a successor of $S_i$.  Treating
subscripts modulo $k$, for $1\le i\le k$ add an edge $t_is_{i+1}$ such that
$t_i\in V(T_i)$ and $s_{i+1}\in V(S_{i+1})$.

The resulting digraph $G'$ is weakly connected.  It is obviously strict when
$k>2$, and when $k=2$ the edges $t_2s_1$ and $t_1s_2$ do not have the same
pair of endpoints because at least one weak component is not strong.  For the
same reason, the $k$ added edges are too few to complete a complete dicut unless
$k=2$ and $G$ has exactly three vertices, but then $G'$ is a (directed) cycle.  
Thus, the connected case applies to $G'$.

Let $r'$ be the number of strong components of $G'$.  Since all $S_i$ and all
$T_i$ lie in one strong component in $G'$, and there are at least $k+1$ such
sets since $S_i=T_i$ only when $G_i$ is strong, we have $r' \le r-k$.  By
Case 1, $G'$ can be made strongly connected by adding at most $r'-1$ edges, so
the number of edges needed to make $G$ strongly connected is at most $r-1$.
\end{proof}

This proof appears to require that, in the situation of Boesch--Tindell and Farzad et al., the graph underlying the partially oriented graph $M$ is the complete graph $K_n$.  Therefore, we do not expect a similar bound in the generality of those papers.

\begin{remark}\label{lbound}
Since every source component needs an entering edge and every sink component
needs an exiting edge, at least $\max(s,t)$ added edges are needed for a
strongly connected extension, where $s$ and $t$ are the numbers of source and
sink components.
\end{remark}

\begin{example}\label{usharp-conn}
The upper bound $r-1$ is sharp for weakly connected strict digraphs.  
Consider the digraph obtained from a transitive tournament with $r$ vertices by
deleting the unique spanning path.  There are $r$ strong components and no
complete dicut, and extending to a strong strict digraph requires adding the
$r-1$ missing edges.  When $r\ge4$, this example has two source components and
two sink components, so the lower bound in Remark~\ref{lbound} can be
arbitrarily bad when $G$ is weakly connected.
\end{example} 

\begin{example}\label{usharp-disc}
The upper bound $r-1$ is also sharp for digraphs with more than one weak
component when the components are not all strong.  Let $G$ be the digraph
formed from the complete bipartite graph $K_{p,q}$ with bipartition
$(X,Y)$, where $|X|=p$ and $|Y|=q$, by directing each edge from $X$ to $Y$
and adding an isolated vertex.  A strong extension must add
edges entering each vertex in $X$ and edges leaving each vertex in $Y$, but no
edge can do both.  Thus, the needed number of added edges is at least $p+q$,
which equals $r-1$.  Here again the bound of Remark~\ref{lbound} is weak.
\end{example}

In spite of Example~\ref{usharp-disc}, the upper bound can be reduced in the
disconnected case by using additional information, such as the number of
weak components that are not strong, the number of strong components that are
neither sources nor sinks, or the number of source and sink components in each
weak component.  We omit the details of the first two; the next result concerns
the last.

\begin{lemma}[Disconnected upper bound]\label{dubound-sum}
Let $G$ be a strict digraph that is not weakly connected.  Let $G_1,\ldots,G_k$
be its weak components, and let $s_i$ and $t_i$ be the number of source and
sink components, respectively, in $G_i$.  The minimum number of edges that must
be added to make $G$ strongly connected is at most
$
\max(t_1,s_2) + \cdots + \max(t_{k-1},s_k) + \max(t_k,s_1).
$
\end{lemma}

\begin{proof}
We add edges from sink components of $G_{i-1}$ to source components of $G_i$,
viewing subscripts modulo $k$.  We add at least one edge leaving each sink
component and one edge entering each source component.  This is easy to do
using $\max(t_{i-1},s_i)$ edges.  The resulting digraph is obviously strongly
connected and, when $k\ge3$, strict.  When $k=2$, the same observation as in
Case 2 of Theorem~\ref{ubound} allows it to be strict.
\end{proof}

The exact value of this upper bound depends on the cyclic order chosen for the
components of $G$.  We do not know a procedure to minimize the sum.
The lower bound from Remark~\ref{lbound} is $\max(\sum t_i,\sum s_i)$; it
equals the upper bound from Lemma~\ref{dubound-sum} when the comparison of
$t_{i-1}$ and $s_i$ goes the same way for each $i$.  Hence both bounds are
sharp.

Recall that the weak components of $G$ correspond to the components of $\hat G$.

\begin{prop}[Disconnected upper bound]\label{dubound}
A strict digraph $G$ that is not weakly connected can be strongly connected by
adding at most $s+t-c$ edges, where $G$ has $s$ source components, $t$ sink
components, and $c$ weak components.  The bound equals $u-c'$, where $u$ strong
components are source or sink components and $c'$ weak components are not
strong components.
\end{prop}

\begin{proof}
This follows from Lemma~\ref{dubound-sum}, since
$\max(t_{i-1},s_i)\le t_{i-1}+s_i-1$ (again taking subscripts modulo $k$),
giving an upper bound of $\sum t_{i-1}+\sum s_i-c$.  The sum is $s+t-c$.  The
sum $s+t$ exceeds $u$ by the number $c-c'$ of weak components that are strongly
connected; hence $s+t-c=u-c'$.
\end{proof}

Proposition~\ref{dubound} strengthens the upper bound in Theorem~\ref{ubound}
when $G$ is not weakly connected.  By definition, always $u\le r$, so the upper
bound of $r$ is reduced by at least $1$ for each weak component that is not
strong.

\begin{example}
Toward understanding the problem of obtaining strong extensions by adding the
fewest edges, it is interesting to consider digraphs obtained from 
bipartite graphs.  Let $G$ be a strict digraph obtained from a bipartite graph
with bipartition $(X,Y)$ by orienting all edges from $X$ to $Y$.  We forbid the
underlying graph to be complete bipartite, because the orientation would yield
a complete dicut.  Let $s=|X|$ and $t=|Y|$.

As in Remark~\ref{lbound}, a strongly connected extension of $G$ must always
add at least $\max(s,t)$ edges; this is the trivial lower bound.  By symmetry,
suppose $s\ge t$.  We claim that achieving this lower bound requires adding
edges not in $\hat G$ that match $Y$ into $X$.  To
see this, note that $s$ added edges must enter $X$ and $t$ added edges must
exit $Y$.  To do this using only $s$ edges, each of the $t$ added edges leaving
$Y$ must be one of the $s$ edges entering $X$.  These $t$
edges form a matching of $Y$ into $X$.

If the digraph on $2t$ vertices induced by the vertices covered by this
matching is strongly connected, then adding an edge from the matched vertices
of $X$ to each of its $s-t$ remaining vertices completes the desired strongly
connected extension.

At the other extreme, when $G$ has $st-1$ edges, only the one edge missing
from $K_{s,t}$ can connect a sink to a source, and the upper bound $s+t-1$
cannot be improved.  

The common generalization of the two extremes improves
the lower bound to $s+t-m$, where $m$ is the maximum size of a matching from
$Y$ into $X$ using edges not in the original bipartite graph.
\end{example}

A digraph $G$ is \emph{$k$-connected} if it has more than $k$ vertices and any
deletion of fewer than $k$ vertices from $G$ leaves a strong digraph.
For a generalization analogous to those studied for non-strict digraphs,
we say that a strict digraph $G$ is \emph{$k$-connectable} if it extends
to a $k$-connected strict digraph.  An obvious necessary condition for
$k$-connectability is that every dicut $[X,\Xb]$ lacks at least $k$ edges.
It is not clear whether this condition is sufficient; our proofs for $k=1$
do not extend, because when $k\ge2$ the maximal $k$-connected subgraphs 
of a graph need not be pairwise disjoint.

\section{Non-transitive dice}\label{secdice}

In a set of ordinary dice, all dice are the same and the probability that one
die rolls a higher number than another is, if we exclude ties (e.g., by ignoring them and rolling again), exactly $1/2$.  Martin Gardner
publicized the idea, due to Bradley Efron, of dice where not only is the
probability other than $1/2$, but there can be three dice such that
each beats one of the others with probability greater than $1/2$
(see Gardner~\cite{G, G2, G3}).  Such dice are \emph{non-transitive}: generalizing to
$n$ dice, there exist $l$ that can be arranged cyclically so that each has
probability greater than $1/2$ of rolling higher than its successor.

Quimby found a set of four non-transitive 6-sided dice whose 24 sides are the
distinct numbers from 1 to 24~\cite{Q} (see Savage~\cite{Sa} for other sets of
non-transitive dice).  Schaefer and Schweig~\cite{2S} carried the idea further;
they studied $k$-sided generalized dice in which each die has $k$ different
numbers on it and the numbers on all dice are distinct.  We may regard each die
as a set of $k$ distinct integers and define a \emph{set of dice} $D$ as a set
of pairwise disjoint such sets.  As they observed, one can always choose the
dice to partition the set $\{1,\ldots,kn\}$.)

We say that die $D_1$ \emph{beats} $D_2$ and write $D_1\succ D_2$ if, among all
pairs of numbers in $D_1 \times D_2$, the first number is larger than the
second more than half the time.  A set of dice is \emph{transitive} if the
relation $\succ$ is transitive.  Although it may be contrary to intuition, a
randomly chosen set of (more than two) dice need not be transitive.  Schaefer
and Schweig~\cite{2S} showed that it is easy to make an intransitive set of
three or four $k$-sided dice when $k\ge3$.

The relation $\succ$ can be represented by a digraph $G(D)$ with one vertex for
each die and an edge from $i$ to $j$ if $D_j$ beats $D_i$.  The relation
$\succ$ is antisymmetric, meaning that $D\succ D'$ and $D'\succ D$ imply
$D=D'$, so the digraph $G(D)$ is strict.  If $H$ is any subgraph of $G(D)$, we
say that $D$ \emph{realizes} $H$; that means every edge of $H$ corresponds to a
pair of dice in $D$ in which one beats the other as indicated by the direction
of the edge.  ($H$ need not be an induced subgraph.)

Let $p_{i,j}$ be the probability that $D_i$ beats $D_j$; that is, $p_{i,j}$ is
the proportion of pairs in $D_i \times D_j$ in which the first number is the
larger.  (Because no two numbers on dice are equal, $p_{i,j}+p_{j,i}=1$.)  
Schaefer and Schweig~\cite{2S} call a set of dice \emph{balanced} if all
unordered pairs $\{p_{i,j},p_{j,i}\}$ are the same.  They found that for
$n=3\le k$ it is possible to form non-transitive dice that are balanced.
Schaefer~\cite{S} then showed that for a tournament $T$ of order $n \ge3$,
there is a balanced set of $n$ non-transitive dice that realizes $T$ if and
only if $T$ is strongly connected.  He observed the corollary that
non-transitive dice realizing a strict digraph $G$ can be chosen balanced if
and only if $G$ is strongly connectable.  Thus Theorem \ref{main} gives a
criterion for the existence of balanced dice realizing a given relation
(transitive or not).

\begin{corollary}\label{dice}
An antisymmetric relation is realizable by a set of balanced dice if and only
if its digraph has no complete dicut.
\end{corollary}

\end{document}